\documentclass{amsart}
\usepackage{amssymb}

\theoremstyle{plain}
\newtheorem{theorem}{Theorem}[section]
\newtheorem{lemma}[theorem]{Lemma}

\newtheorem{corollary}[theorem]{Corollary}

\newtheorem{example}[theorem]{Example}
\theoremstyle{definition}

\numberwithin{equation}{section}

\newcommand{\thmlabel}[1]{\label{thm:#1}} 
\newcommand{\lemlabel}[1]{\label{lem:#1}} 
\newcommand{\corlabel}[1]{\label{cor:#1}} 
\newcommand{\seclabel}[1]{\label{sec:#1}} 
\newcommand{\eqnlabel}[1]{\label{eqn:#1}} 

\newcommand{\thmref}[1]{\ref{thm:#1}}           
\newcommand{\lemref}[1]{\ref{lem:#1}}           
\newcommand{\corref}[1]{\ref{cor:#1}}           
\newcommand{\eqnref}[1]{\eqref{eqn:#1}}     

\newcommand{\sblp}[1]{\langle #1 \rangle}   
\newcommand{\OR}{\mbox{ or }}
\newcommand{\AND}{\mbox{ and }}
\newcommand{\define}[1]{\emph{#1}}

\newcommand{\Jordan}{\textsc{J}}

\title{Powers of elements in Jordan loops}

\author[K. Pula]{Kyle Pula}
\email[Pula]{jpula@math.du.edu}

\address{Department of Mathematics \\
University of Denver \\
2360 S Gaylord St \\
Denver, CO 80208, U.S.A.}

\begin{document}



\begin{abstract}
A Jordan loop is a commutative loop satisfying the Jordan identity 
$(x^2 y ) x = x^2 ( y x)$. We establish several identities involving powers in Jordan loops
and show that there is no nonassociative Jordan loop of order $9$.
\end{abstract}
\maketitle

\section{Introduction}

A magma $(Q,\cdot)$ is a \emph{quasigroup} if, for each $a,b\in Q$,
the equations $ax=b \AND ya=b$ have unique solutions $x,y\in Q$. A
\emph{loop} is a quasigroup with a neutral element, which we denote $e$. Standard
references on quasigroup and loop theory are \cite{Bruck,pflug}. 
A commutative loop is said to be \emph{Jordan} if it satisfies
the \emph{Jordan identity}
\begin{equation}\eqnlabel{Jordan}
    x^2(yx) = (x^2y)x   \tag{\Jordan}.
\end{equation}

Kinyon, Pula, and Vojt\v{e}chovsk\'y \cite{Admissible_Orders} showed that 
there exists a nonassociative (that is, not associative) Jordan loop
of order $n$ if and only if $n \geq 6$ and $n \neq 9$. 

For the order $9$ case, their work relied upon
an exhaustive computer search. In this paper, we establish several
identities involving powers in Jordan loops and present a more
``human-sized'' proof that there are no nonassociative Jordan loops of order $9$.

\section{Powers of Elements}
\seclabel{order9}

We write $x^k$ for the right associated term $L_x ^k (e) = x (x ( \cdots (xe) \cdots ))$. We say that \define{$x^k$ is well-defined} if the value of this term does not depend on how it is associated.

\begin{lemma}
\lemlabel{SmallPowers}
If $Q$ is a Jordan loop and $x \in Q$, then $x^3, x^4, \AND x^5$ are well-defined.
\end{lemma}
\begin{proof}
Third powers are well-defined in any commutative loop. For the fourth power, $x^3 x = x^2 x \cdot x = x^2 \cdot xx = x^2 x^2$. For the fifth power, $x^4 x = x^2 x^2 \cdot x = x^2 \cdot x^2 x = x^2 x^3$.
\end{proof}

\begin{lemma}
\lemlabel{Power2Identities}
The following identities hold in any Jordan loop:
\begin{enumerate}
\item $x^n x^2 = x^{n+2}$
\item $x^n x^4 = x^{n+4}$
\item $x^n x^8 = x^{n+8}$ if $n \not\equiv 3 \mod 4$ or $x^3 x^8 = x^{11}$

\item $x^n x^{2^k} = x^{n + 2^k}$ if $n \equiv 2^m \mod 2^{k-1}$ for $0 \leq m \leq (k-1)$.

\item $x^{2^n} = (x^{2^{n-1}})^2$

\end{enumerate}
\end{lemma}

\begin{proof}
(i) This is trivial for $n=0$. Assuming the identity holds for $n-1$ and using \eqnref{Jordan}, $x^n x^2 = x x^{n-1} \cdot x^2 = x^2 x^{n-1} \cdot x = x^{n+1} x = x^{n +2}$.

(ii) This is trivial for $n=0 \AND n=1$. Assuming the identity holds for $n-2$ and using (i) and \eqnref{Jordan}, $x^n x^4 = x^{n-2} x^2 \cdot x^4 = x^{n-2} x^4 \cdot x^2 = x^{n+2} x^2 = x^{n+4}$.

(iii) This is trivial for $n=0 \AND n=1$ while $n=2$ follows from (i) and $n=3$ holds by assumption. Assuming the identity holds for $n-4$, $x^n x^8 = x^{n-4} x^4 \cdot x^8 = x^{n-4} x^8 \cdot x^4 = x^{n+4} x^4 = x^{n+8}$, using (ii) and \eqnref{Jordan}. By induction, the identity holds for all $n \not \equiv 3 \mod 4$ and if the identity holds for $n=3$, then it holds for all $n$.

(iv) We say J$(n,k)$ holds if (iv) holds for $n$ and $k$. Note that for $k=1,2, \AND 3$, J$(n,k)$ is a special case of (i), (ii), and (iii). Assume that J$(m,i)$ holds for all $m$ and for all $i < k$ and consider J$(n,k)$. For $n=2^m$ where $0 \leq m \leq (k-1)$, the identity J$(n,k)$ is $x^{2^m} x^{2^k} = x^{2^m + 2^k}$ but, in the presence of commutativity, this identity is also J$(2^k,m)$. Since $m < k$, J$(2^k,m)$ holds by our induction assumption.

We now keep $k$ fixed and induct on $n$. Assume that $n \equiv 2^m \mod 2^{k-1}$ for $0 \leq m \leq (k-1)$ and that J$(n-2^{k-1}, k)$ holds. Note that if $n \equiv 2^m \mod 2^{k-1}$, then $n-2^{k-1} \equiv 2^m \mod 2^{k-2}$ and thus it follows from J$(n-2^{k-1},k-1)$ that $x^n = x^{n-2^{k-1}} x^{2^{k-1}}$ and by J$(2^{k-1}, k-1)$, we have $x^{2^k} = x^{2^{k-1}} x^{2^{k-1}} = (x^{2^{k-1}})^2$.

Therefore, we have:
\begin{align*}
x^n x^{2^k} 
&= x^{n-2^{k-1}} x^{2^{k-1}} \cdot (x^{2^{k-1}})^2 & \mbox{J}(n-2^{k-1},k-1) \AND \mbox{J}(2^{k-1}, k-1) \\
&= x^{n-2^{k-1}} (x^{2^{k-1}})^2 \cdot x^{2^{k-1}} & \eqnref{Jordan} \\
&= x^{n-2^{k-1}} x^{2^{k}} \cdot x^{2^{k-1}} & \mbox{J}(2^{k-1}, k-1) \\
&= x^{n+2^{k-1}} x^{2^{k-1}} & \mbox{J}(n - 2^{k-1}, k) \AND n-2^{k-1} \equiv 2^m \mod 2^{k-1} \\
&= x^{n+2^k}
\end{align*}
The final line follows since J$(n+2^{k-1}, k-1)$ holds and $n+2^{k-1} \equiv 2^m \mod 2^{k-2}$.

(v) This is just the identity J$(2^{n-1}, n-1)$, which applies since $2^{n-1} \equiv 0 \mod 2^{n-2}$. 
\end{proof}

\begin{corollary}
If $Q$ is a Jordan loop and $x \in Q$, then $x^{n} = x^{1 \cdot a_0} ( x^{2 \cdot a_1} ( \cdots (x^{2^k \cdot a_k}) )$ where $a_k \ldots a_0$ is the binary expansion of $n$.
\end{corollary}

\begin{example} The following identity holds in any Jordan loop:
\begin{align*}
x^{317} = x^{(100111101)_2} = x(x^4(x^8(x^{16}(x^{32}(x^{256})))))
\end{align*}
\end{example}

\begin{lemma}
\lemlabel{InverseIdentities}
The following identities hold in any Jordan loop:
\begin{enumerate}
\item $x^2 x^{-1} = x$
\item $x^4 x^{-1} = x^3$ 
\item $x^8 x^{-1} = x^7$ if $x^3 x^8 = x^{11}$ 
\end{enumerate}
\end{lemma}

\begin{proof}
(i) By \eqnref{Jordan}, $x^2 = x^2 \cdot x x^{-1} = x \cdot x^2 x^{-1}$ and we may now cancel $x$ from both sides to get $x = x^2 x^{-1}$.

(ii) Recall that $x^4 = (x^2)^2$. By \eqnref{Jordan} and (i), $x^2 \cdot x^4 x^{-1} = x^4 \cdot x^2 x^{-1} = x^4 x = x^2 x^3$ and we may now cancel $x^2$ from both sides to get $x^4 x^{-1} = x^3$.

(iii) Recall that $x^8 = (x^4)^2$. By \eqnref{Jordan} and (ii), $x^4 \cdot x^8 x^{-1} = x^8 \cdot x^4 x^{-1} = x^8 x^3 = x^{11} = x^4 x^7$ and we may cancel $x^4$ from both sides to get $x^8 x^{-1} = x^7$.

\end{proof}

\begin{lemma}
\lemlabel{Inverse}
If $Q$ is a Jordan loop and $x \in Q$, then $(x^{2^n})^{-1} = (x^{-1})^{2^n}$.
\end{lemma}
\begin{proof}
The identity is trivial for $n=0$. For $n=1$, we have
\begin{align*}
(x^{-1})^2
&= (x^{-1})^2 \cdot x x^{-1} \\
&= (x^{-1})^2 x \cdot x^{-1} & \eqnref{Jordan} \\
&= (x^{-1})^2 (x^2 x^{-1}) \cdot x^{-1} & \mbox{(i) of Lemma \lemref{InverseIdentities}} \\
&= ((x^{-1})^2 x^2) x^{-1} \cdot x^{-1} & \eqnref{Jordan}
\end{align*}
Cancel $x^{-1}$ from both sides twice to get $e = (x^{-1})^2 x^2$. Thus $(x^{-1})^2 = (x^2)^{-1}$. Now assuming the identity holds for $n-1$, we have
\begin{align*}
(x^{2^n})^{-1}
&= ((x^{2^{n-1}})^2)^{-1} & \mbox{(v) of Lemma \lemref{Power2Identities}} \\
&= ((x^{2^{n-1}})^{-1})^2 & \mbox{Previous Case} \\
&= ((x^{-1})^{2^{n-1}})^2 & \mbox{Induction Assumption} \\
&= (x^{-1})^{2^{n}} & \mbox{(v) of Lemma \lemref{Power2Identities}}
\end{align*}

\end{proof}

\begin{lemma}
\lemlabel{Identities2}
The following identities hold in any Jordan loop:
\begin{enumerate}
\item $(x^2)^{-1} x = x^{-1}$
\item $x^3 x^{-2} = x$ 
\item $x^3 x^{-1} = x^2$ 
\item $x^4 (x^{-1})^3 = x$ 
\item $x^6 x^{-2} = x^4$ 
\item $x^6 x^{-4} = x^2$ 
\end{enumerate}
\end{lemma}

\begin{proof}
(i) Let $y = x^{-1}$. Then $(x^2)^{-1} x = (x^{-1})^2 x = y^2 y^{-1} = y = x^{-1}$, using Lemma \lemref{Inverse} and (i) of Lemma \lemref{InverseIdentities}.

(ii) First, $x^3 x^{-2} = x^4 x^{-1} \cdot x^{-2} = x^4 x^{-1} \cdot (x^{-1})^2 = x^4 (x^{-1})^2 \cdot x^{-1}$, using (ii) of Lemma \lemref{InverseIdentities}. Let $y = x^2$ then $x^4 (x^{-1})^2 \cdot x^{-1} = y^2 y^{-1} \cdot x^{-1} = y x^{-1} = x^2 x^{-1} = x$, using (i) of Lemma \lemref{InverseIdentities} twice.

(iii) Using \eqnref{Jordan} and (ii), $x^{-2} \cdot x^{-1} x^{3} = x^{-1} \cdot x^{-2} x^{3} = x^{-1} x = e$. Thus $x^3 x^{-1} = (x^{-2})^{-1} = x^2$.

(iv) Let $y = x^{-1}$. Then $x^2 \cdot x^4 (x^{-1})^3 = x^4 \cdot x^2 (x^{-1})^3 = x^4 \cdot y^{-2} y^3 = x^4 y = x^3$, using \eqnref{Jordan}, (ii), and (i) of Lemma \lemref{InverseIdentities}. Now cancel $x^2$ from both sides to get $x^4 (x^{-1})^3 = x$.

(v) Using (i) of Lemma \lemref{Power2Identities}, Lemma \lemref{Inverse}, and (iii), $x^6 x^{-2} = (x^2)^3 (x^2)^{-1} = (x^2)^2 = x^4$.

(vi) Using (i) of Lemma \lemref{Power2Identities}, Lemma \lemref{Inverse}, and (ii), $x^6 x^{-4} = (x^2)^3 (x^2)^{-2} = x^2$.
\end{proof}

\begin{theorem}
\thmlabel{8WellDefined}
If $Q$ is a Jordan loop and $x \in Q$ such that $x^3 x^3 = x^6$, then
\begin{enumerate}
\item $x^6$ is well-defined
\item $x^7$ is well-defined
\item $x^6 x^{-1} = x^5$
\item $x^8$ is well-defined
\end{enumerate}
\end{theorem}
\begin{proof}
(i) $x^3 x^3 = x^6 = x x^5 = x \cdot x^2 x^3 = x^2 \cdot x x^3 = x^2 x^4$.

(ii) $x^6 x = x^2 x^4 \cdot x = x^2 x^5 = x^4 x \cdot x^2 = x^4 x^3$.

(iii) $x^6 x^{-1} = (x^3)^2 \cdot x^3 x^{-4} = x^3 \cdot (x^3)^2 x^{-4} = x^3 \cdot x^6 x^{-4} = x^3 x^2 = x^5$.

(iv) $x^8 = x^6 x^2 = (x^3)^2 \cdot x^3 x^{-1} = x^3 \cdot (x^3)^2 x^{-1} = x^3 x^5$.

\end{proof}
Theorem \thmref{NonWellDefinedPowers} shows that Theorem \thmref{8WellDefined} cannot be improved.

\begin{theorem}
\thmlabel{NonWellDefinedPowers}
If $n > 5$ is neither a power of two nor prime, then there is a Jordan loop $Q$ and a generating element $x \in Q$ such that $x^k$ is well-defined for $0 \leq k < n$ but $x^n$ is not well-defined.
\end{theorem}
  
\begin{proof}
See Theorem 5.5 of \cite{Admissible_Orders}.
\end{proof}

\section{Jordan Loops of Order 9}

The following is a well-known and simple result. We reproduce it here for completeness.

\begin{lemma}
\lemlabel{EvenOrder}
A commutative loop $Q$ of order $n$ has a nontrivial involution if and only if $n$ is even.
\end{lemma}
\begin{proof}
Fix a multiplication table for $Q$. Note that every element of $Q$ appears in the multiplication table $n$ times. Since $Q$ is commutative, every element appears the same number of times above the main diagonal as it does below. Thus every element appears an even number of times off the main diagonal. If $n$ is even, then every element must appear an even number of times on the main diagonal while if $n$ is odd, every element must appear an odd number of times on the main diagonal.

Thus, if $n$ is odd, then every element must appear exactly once on the main diagonal. In particular, since $e$ must appear in the cell corresponding to $e \cdot e$, it cannot appear anywhere else. If $n$ is even, since $e$ must appear in the $e \cdot e$ cell, it must also appear somewhere else along the main diagonal.
\end{proof}

\begin{corollary}
\corlabel{EvenOrderedSubLoops}
A commutative loop $Q$ of order $n$ has an even-ordered subloop if and only if $n$ is even.
\end{corollary}

\begin{corollary}
\corlabel{SquareRoot}
A commutative loop $Q$ of order $n$ has a well-defined square root operation if and only if $n$ is odd.
\end{corollary}

\begin{lemma}
\lemlabel{MaxSubSquareSize}
If $H$ is a proper subquasigroup of a finite quasigroup $Q$, then $|H| \leq \lfloor \frac{|Q|}{2}\rfloor$.
\end{lemma}
\begin{proof}
Let $k = |H|$ and $n+k = |Q|$. Fix a multiplication table of $Q$ with both the rows and columns indexed first by elements of $H = \{h_i\}$ and then of $Q \setminus H = \{q_i\}$. Since $H$ is a subquasigroup of $Q$, the cells corresponding to $H \times H$ contain only elements of $H$. Then the $k$ cells corresponding to $q_1 \times H$ must be filled entirely with elements from $Q \setminus H$ and thus $Q \setminus H = n \geq k$. That is, $n+k = |Q| \geq 2|H|$ and thus $|H| \leq \lfloor \frac{|Q|}{2}\rfloor$.
\end{proof}

\begin{lemma}
\lemlabel{Monogenic} Let $Q$ be a loop of order $n$ and let $x\in Q$. If $x^m$
is well-defined for every $1\le m\le n-1$ then $\sblp{x}$ is a cyclic group of
order $k$, and $k=n$ whenever $k>\lfloor n/2\rfloor$.
\end{lemma}

\begin{proof}
See Lemma 5.3 in \cite{Admissible_Orders}.
\end{proof}

\begin{lemma}
\lemlabel{Order9Cases}
If $Q$ is a Jordan loop of order 9, then $Q$ is either of exponent 3 or cyclic.
\end{lemma}
\begin{proof}
Suppose $e \neq x \in Q$ does not generate $Q$ and let $k = |\sblp{x}|$. Lemma \lemref{MaxSubSquareSize} shows that $k \leq \lfloor 9 \rfloor = 4$ and Corollary \corref{EvenOrderedSubLoops} shows that $k = 3$.
\end{proof}

\begin{lemma}
\lemlabel{Order9MonogenicPowers}
If $Q = \sblp{x}$ is a Jordan loop of order 9, then $Q = \{x^k : 1 \leq k \leq 9\}$ and $x^n = x^{(n \mod 9)}$ for all $n \geq 0$.
\end{lemma}
\begin{proof}
Suppose $x^{n} = x^{n + k}$ for $1 \leq n < n+k \leq 9$. Cancel terms on the left to get $e = x^k$ and consider the smallest possible value of $k$. It is easy to see that if $k = 1,2, \OR 3$, then $|\sblp{x}| = k$, a contradiction. If $k=4$, then $x^3 x^3 = x^3 x^{-1} = x^2 = x^6$ and thus $x^6$ is well-defined. It then follows that $|\sblp{x}| = 4$, a contradiction. If $k=5$, then $x^3 x^3 = x^3 (x^2)^{-1} = x^3 x^{-2} = x = x^6$ and thus $x^6$ is well-defined. Again it follows that $|\sblp{x}| = 5$, a contradiction.

Suppose $k=6$. Multiplying $x^2$ on both sides of $x^6 = e$ gives $x^8 = x^2$. Taking the square root of both sides gives $x^4 = x$ and thus $x^3 = e$, a contradiction.

Suppose $k=7$. We show that $x^n$ is well-defined for all $n$ and by Lemma \lemref{Monogenic}, $\sblp{x}$ is a cyclic group of order $7$, a contradiction. Since $x^7 = x^3 x^4 = e$, $x^3 = (x^4)^{-1}$. Squaring both sides and applying Lemma \lemref{Inverse}, $x^3 x^3 = (x^4 x^4)^{-1} = (x^8)^{-1} = x^{-1} = x^6$. We now have that $x^6$ is well-defined and by Theorem \thmref{8WellDefined} we are done.

Suppose $k=8$. Then $x^8 = x^4 x^4 = e$ and by Lemma \lemref{EvenOrder} $x=e$.

We thus have that $x^9 = e$. Fix $n \geq 9$ and note that $x^n = x \cdot x^{n-1}$. By induction $x^{n-1} = x^{(n-1 \mod n)}$. Thus $x^n = x^{(n \mod 9)}$.

\end{proof}

\begin{lemma}
\lemlabel{Order9Monogenic}
If $Q$ is a cyclic Jordan loop of order 9, then $Q$ is a group.
\end{lemma}
\begin{proof}
Let $\sblp{x} = Q$. By Lemma \lemref{Monogenic}, we will be done if we show that $x^k$ is well-defined for $1 \leq k \leq 8$. By Lemma \lemref{SmallPowers} and Theorem \thmref{8WellDefined}, we only need to consider $k=6$. By Lemma \lemref{Order9MonogenicPowers}, we may write every element of $Q$ as $x^k$ for $0 \leq k \leq 8$. We now use Lemma \lemref{Power2Identities} to fill in a partial multiplication table for $Q$ as in Table \ref{TblCyclic9}.
\begin{table}[htdp]
\begin{center}
$
\begin{array}{|c|c|c|c|c|c|c|c|c|}
\hline e   & x   & x^2 & x^3 & x^4 & x^5 & x^6 & x^7 & x^8 \\
\hline x   & x^2 & x^3 & x^4 & x^5 & x^6 & x^7 & x^8 & e   \\
\hline x^2 & x^3 & x^4 & x^5 & x^6 & x^7 & x^8 & e   & x   \\
\hline x^3 & x^4 & x^5 &     & x^7 &     &     &     &     \\
\hline x^4 & x^5 & x^6 & x^7 & x^8 & e   & x   & x^2 & x^3 \\
\hline x^5 & x^6 & x^7 &     & e   &     &     &     & x^4 \\
\hline x^6 & x^7 & x^8 &     & x   &     &     &     & x^5 \\
\hline x^7 & x^8 & e   &     & x^2 &     &     &     &     \\
\hline x^8 & e   & x   &     & x^3 & x^4 & x^5 &     & x^7 \\
\hline \end{array}
$
\end{center}
\caption{Partial multiplication table for $Q$}
\label{TblCyclic9}
\end{table}

Since values cannot repeat in columns, rows, or the main diagonal, $x^3 x^3 = x$ or $x^3 x^3 = x^6$. In the latter case, $x^6$ is well-defined and we are done. Suppose $x^3 x^3 = x$ and note that $(x^3 x^3) x^3 \cdot x^3 = x x^3 \cdot x^3 = x^7$, but by \eqnref{Jordan}, $(x^3 x^3) x^3 \cdot x^3 = x^3 x^3 \cdot x^3 x^3 = x \cdot x = x^2$. Thus $x^7 = x^5$ and $x^2 = e$, a contradiction.

\end{proof}

\begin{theorem}
If $Q$ is a Jordan loop of order $9$, then $Q$ is a group.
\end{theorem}
\begin{proof}
By Lemmas \lemref{Order9Cases} and \lemref{Order9Monogenic}, we only need to consider the case where $Q$ is of exponent $3$. Let $e \neq a,b,c,d \in Q$ such that $\sblp{a}, \sblp{b}, \sblp{c}, \AND \sblp{d}$ are distinct. A partial multiplication table for $Q$ must be of the form presented in Table (A).

\begin{table}[htdp]
\begin{center}

\begin{tabular}{c c}
$
\begin{array}{|l|l l|l l|l l|l l|}
\hline 
e   & a   & a^2 & b   & b^2 & c   & c^2 & d   & d^2 \\ \hline
a   & a^2 & e   &     &     &     &     &     &     \\ 
a^2 & e   & a   &     &     &     &     &     &     \\ \hline
b   &     &     & b^2 & e   &     &     &     &     \\ 
b^2 &     &     & e   & b   &     &     &     &     \\ \hline
c   &     &     &     &     & c^2 & e   &     &     \\ 
c^2 &     &     &     &     & e   & c   &     &     \\ \hline
d   &     &     &     &     &     &     & d^2 & e   \\ 
d^2 &     &     &     &     &     &     & e   & d   \\ \hline
\end{array}
$
&
$
\begin{array}{|c|c c|c c|c c|c c|}
\hline 
e   & a   & a^2 & b   & b^2 & c   & c^2 & d   & d^2 \\ \hline
a   & a^2 & e   & c   &     &     &     &     &     \\ 
a^2 & e   & a   &     & c   &     &     &     &     \\ \hline
b   & c   &     & b^2 & e   &     &     &     &     \\ 
b^2 &     & c   & e   & b   &     &     &     &     \\ \hline
c   &     &     &     &     & c^2 & e   &     &     \\ 
c^2 &     &     &     &     & e   & c   &     &     \\ \hline
d   &     &     &     &     &     &     & d^2 & e   \\ 
d^2 &     &     &     &     &     &     & e   & d   \\ \hline
\end{array}
$
\\\\
Table (A) & Table (B)
\\
\\
$
\begin{array}{|l|l l|l l|l l|l l|}
\hline 
e   & a   & a^2 & b   & b^2 & c   & c^2 & d   & d^2 \\ \hline
a   & a^2 & e   & c   & c^2 &     &     &     &     \\ 
a^2 & e   & a   & d   & d^2 &     &     &     &     \\ \hline
b   & c   & d   & b^2 & e   &     & a   &     &     \\ 
b^2 & c^2 & d^2 & e   & b   & a   &     &     &     \\ \hline
c   &     &     &     & a   & c^2 & e   &     &     \\ 
c^2 &     &     & a   &     & e   & c   &     &     \\ \hline
d   &     &     &     &     &     &     & d^2 & e   \\ 
d^2 &     &     &     &     &     &     & e   & d   \\ \hline
\end{array}
$
&
$
\begin{array}{|l|l l|l l|l l|l l|}
\hline 
e   & a   & a^2 & b   & b^2 & c   & c^2 & d   & d^2 \\ \hline
a   & a^2 & e   & c   & d^2 &     & x^2 & x   &     \\ 
a^2 & e   & a   & d   & c^2 & x   &     &     & x^2 \\ \hline
b   & c   & d   & b^2 & e   &     & y^2 &     & y   \\ 
b^2 & d^2 & c^2 & e   & b   & y   &     & y^2 &     \\ \hline
c   &     & x   &     & y   & c^2 & e   &     &     \\ 
c^2 & x^2 &     & y^2 &     & e   & c   &     &     \\ \hline
d   & x   &     &     & y^2 &     &     & d^2 & e   \\ 
d^2 &     & x^2 & y   &     &     &     & e   & d   \\ \hline
\end{array}
$
\\\\
Table (C) & Table (D)
\\
\\
$
\begin{array}{|l|l l|l l|l l|l l|}
\hline 
e   & a   & a^2 & b   & b^2 & c   & c^2 & d   & d^2 \\ \hline
a   & a^2 & e   & c   & d^2 & d   & x^2 & x   & c^2 \\ 
a^2 & e   & a   & d   & c^2 & x   & d^2 & c   & x^2 \\ \hline
b   & c   & d   & b^2 & e   & d^2 & y^2 & c^2 & y   \\ 
b^2 & d^2 & c^2 & e   & b   & y   & d   & y^2 & c   \\ \hline
c   & d   & x   & d^2 & y   & c^2 & e   & x^2 & y^2 \\ 
c^2 & x^2 & d^2 & y^2 & d   & e   & c   & y   & x   \\ \hline
d   & x   & c   & c^2 & y^2 & x^2 & y   & d^2 & e   \\ 
d^2 & c^2 & x^2 & y   & c   & y^2 & x   & e   & d   \\ \hline
\end{array}
$
&
$
\begin{array}{|l|l l|l l|l l|l l|}
\hline 
e   & a   & a^2 & b   & b^2 & c   & c^2 & d   & d^2 \\ \hline
a   & a^2 & e   & c   & d^2 & d   & b^2 & b   & c^2 \\ 
a^2 & e   & a   & d   & c^2 & b   & d^2 & c   & b^2 \\ \hline
b   & c   & d   & b^2 & e   & d^2 & a^2 & c^2 & a   \\ 
b^2 & d^2 & c^2 & e   & b   & a   & d   & a^2 & c   \\ \hline
c   & d   & b   & d^2 & a   & c^2 & e   & b^2 & a^2 \\ 
c^2 & b^2 & d^2 & a^2 & d   & e   & c   & a   & b   \\ \hline
d   & b   & c   & c^2 & a^2 & b^2 & a   & d^2 & e   \\ 
d^2 & c^2 & b^2 & a   & c   & a^2 & b   & e   & d   \\ \hline
\end{array}
$
\\\\
Table (E) & Table (F)
\end{tabular}

\end{center}
\end{table}

Suppose an off-diagonal $2 \times 2$ square of Table (A) contains a repeated element. Without loss of generality, we are in the case presented in Table (B). While the column indexed by $d$ must contain the element $c$, there is no available row that can contain this occurrence of $c$. Thus, every off-diagonal 2x2 square in Table (A) must contain four distinct elements.

Suppose an off-diagonal $2 \times 2$ square of Table (A) contains both an element and its square in the same row or column. Without loss of generality, we are in the case presented in Table (C). Let $y := c \cdot b^2 = ab \cdot b^2 = a b^2 \cdot b = c^2 b$. Notice that either $y = a \OR y = a^2$. If $y = a$, then the column indexed by $d$ must contain the element $a$ but there are no available rows to contain this occurrence of $a$. Likewise for $y = a^2$.

Thus every off-diangonal $2 \times 2$ square in Table (A) is of the form
\begin{center}
$
\begin{array}{|c|c|}
\hline 
x & y^2 \\ \hline
y & x^2 \\ \hline
\end{array}
$
\end{center}
for $\sblp{x} \neq \sblp{y}$.

Without loss of generality, we can assume that the $(a,a^2) \times (b,b^2)$ square is arranged as in Table (D). Set $x := d a = a^2 b \cdot a = a^2 \cdot ab = a^2 c$ and $y := d^2 b = a b^2 \cdot b = a b \cdot b^2 = c b^2$ as has been done in Table (D). It is easy to see that Table (E) is the unique quasigroup completion of Table (D).

Note that $\{x, x^2\} = \{b,b^2\} \AND \{y, y^2\} = \{a,a^2\}$. Suppose $x = b^2$. Then $d^2 = b^2 a = a^2 c \cdot a = a^2 \cdot a c = a^2 d = c$, a contradiction, and thus $x = b$. Suppose $y = a^2$. Then $c = b^2 d^2 = b^2 \cdot c b = b^2 c \cdot b = a^2 b = d$, a contradiction, and thus $y = a$.

Therefore, Table (F) must be a multiplication table for $Q$. Furthermore, since we only made labeling choices when completing the table, up to isomorphism, this is the only possible multiplication table for a Jordan loop of order 9 and exponent 3. Therefore, it must be the multiplication table for $Z_3 \times Z_3$.

\end{proof}

\bibliographystyle{plain}

\end{document}